\documentclass[14pt]{article}
\usepackage{amsmath}
\usepackage{amsfonts}

\newtheorem{theorem}{Theorem}

\newtheorem{definition}[theorem]{Definition}

\newtheorem{lemma}[theorem]{Lemma}

\newtheorem{remark}[theorem]{Remark}

\newenvironment{proof}[1][Proof]{\noindent\textbf{#1.} }{\ \rule{0.5em}{0.5em}}
\numberwithin{equation}{section}
 
\begin{document}

\title{TEMPERED RELAXATION EQUATION \\  AND
 RELATED GENERALIZED STABLE PROCESSES}
 \author{Luisa Beghin $^1$, Janusz Gajda $^2$}
\date{}
 \maketitle

 \begin{abstract}

Fractional relaxation equations, as well as relaxation functions
time-changed by independent stochastic processes have been widely
studied (see, for example, \cite{MAI}, \cite{STAW} and
\cite{GAR}). We start here by proving that the upper-incomplete
Gamma function satisfies the tempered-relaxation equation (of
index $\rho \in (0,1)$); thanks to this explicit form of the
solution, we can then derive its spectral distribution, which
extends the stable law. Accordingly, we define a new class of
selfsimilar processes (by means of the $n$-times Laplace transform
of its density) which is indexed by the parameter $\rho $: in the
special case where $\rho =1$, it reduces to the stable
subordinator.

Therefore the parameter $\rho $ can be seen as a measure of the
local deviation from the temporal dependence structure displayed
in the standard stable case.

 \medskip


 \smallskip



 \end{abstract}

 \vspace*{-20pt}



 \section{Introduction}\label{sec:1}

The Cole-Cole law, i.e. $\varphi (t)=e^{-\lambda t},$ which
satisfies the ordinary first-order differential equation
\begin{equation}
\frac{d}{dt}\varphi (t)=-\lambda \varphi (t),\qquad \varphi (0)=1,
\label{rr}
\end{equation}%
(for $\lambda ,t\in \mathbb{R}^{+}$), is suited to model
relaxation phenomena with an exponential decay for large values of
$t$. In many real situations, such as, for example, when
considering the electromagnetic properties of some materials, as
well as the rheological models for viscoelastic materials,
fractional relaxation functions, with long-memory, instead of
exponential, decay, have been preferred. In particular, by
replacing in (\ref{rr}) the first derivative with the fractional
Caputo derivative of order $\alpha \in (0,1),$ the fractional
relaxation equation

\begin{equation}
^{C}\mathcal{D}_{t}^{\alpha }\varphi (t)=-\lambda ^{\alpha
}\varphi (t),\qquad \varphi (0)=1,  \label{fr}
\end{equation}%
has been studied. Its solution is proved in \cite{MAI} to coincide with $%
\varphi (t)=E_{\alpha ,1}(-\lambda ^{\alpha }t^{\alpha })$ (see
the definitions of Mittag-Leffler function and Caputo derivative
below): compared to the standard relaxation function, the
fractional relaxation exhibits, for small values of $t,$ a much
faster decay (the derivative tends to $-\infty $ instead of $-1$),
while, for large values of $t$, it shows a much slower decay (an
algebraic decay instead of an exponential decay). In view of its
slow decay, the phenomenon of fractional relaxation is usually
referred to as a super-slow process (see \cite{MAI}). For further
generalizations of fractional relaxation equations, see
\cite{GAR}.

Relaxation driven by the inverse of a tempered-stable subordinator
has been considered in \cite{STAW}. The frequency-domain
expression of the tempered relaxation function is given there,
together with its asymptotic estimates: this is proved to be an
intermediate case between the superslow relaxation and the
exponential one.

We recall that tempered fractional derivatives are defined and
studied in \cite{MEE}, while a different definition, in terms of
shifted derivative, is given in \cite{BEG}. For some details on
tempered and inverse tempered subordinators see also \cite{WYL},
\cite{WYL2}, \cite{GAJ2}.

We start by considering here the tempered relaxation equation
\begin{equation}
\mathcal{D}_{t}^{\lambda ,\rho }\varphi (t)=-\lambda ^{\rho
}\varphi (t),\qquad \varphi (0)=1,  \label{ri}
\end{equation}%
for $\lambda >0$, $t\geq 0$ and $\rho \in (0,1],$ where $\mathcal{D}%
_{t}^{\lambda ,\rho }$ is defined as particular case of the
so-called convolution-type derivatives (see definitions
(\ref{toa}) and (\ref{temp})
below). We prove that it is satisfied by the following relaxation function%
\begin{equation}
\varphi _{\lambda ,\rho }(t)=\frac{\Gamma (\rho ;\lambda
t)}{\Gamma (\rho )}, \label{trf}
\end{equation}%
where $\Gamma \left( \rho ;x\right) \doteq \int_{x}^{+\infty
}e^{-w}w^{\rho -1}dw$ is the so-called upper-incomplete gamma
function with parameter $\rho \in (0,1].$ Note that $\Gamma \left(
\rho ;x\right) $ is well-defined for
any $\rho \in \mathbb{R}$ and it is real-valued when $x\in \mathbb{R}^{+}$.%

In the special case where $\rho =1,$ equation (\ref{ri}) reduces
to the standard relaxation equation (\ref{rr}). The asymptotic
behavior of the solution coincides with that of the model in
\cite{STAW}, but thanks to the explicit form of the solution, we
can obtain here further results: we can prove the completely
monotonicity and integrability of $\varphi _{\lambda ,\rho }$ and
obtain its spectral distribution, which, as we will explain below,
can be considered a generalization of stable laws.

We recall that the Laplace transform of the $\alpha $-stable
subordinator,
i.e. $\Phi (\eta ):=\mathbb{E}e^{-\eta S_{\alpha }(t)}=e^{-\eta ^{\alpha }t}$%
, $t\geq 0,$ satisfies equation (\ref{rr}), with $\lambda =\eta
^{\alpha }.$ We then generalize this result and define a new class
of stochastic processes by means of the $n$-times Laplace
transform of its density, as follows%
\begin{equation}
\mathbb{E}e^{-\sum_{k=1}^{n}\eta _{k}\mathcal{X}_{\alpha ,\rho
}(t_{k})}=\frac{\Gamma \left( \rho ;\sum\limits_{k=1}^{n}\Xi
_{k,n}^{\alpha}\left( t_{k}-t_{k-1}\right) \right) }{\Gamma \left(
\rho \right) },\qquad \eta _{1},...\eta _{n}\geq 0, n\in
\mathbb{N} \label{lp}
\end{equation}%
where $\Xi _{k,n}^{\alpha}:=\left(\sum_{i=k}^{n}\eta
_{i}\right)^{\alpha}$, for $0<t_{1}<....<t_{n},$ for $\alpha \in
(0,1)$ and $\rho \in \left( 0,1 \right].$
In the one-dimensional case (for $n=1$)$,$ the Laplace transform in (%
\ref{lp}) coincides with (\ref{trf}) for $\lambda =\eta ^{\alpha }$, i.e.%
\begin{equation}
\mathbb{E}e^{-\eta \mathcal{X}_{\alpha ,\rho }(t)}=\frac{\Gamma
\left( \rho ;\eta ^{\alpha }t\right) }{\Gamma \left( \rho \right)
},\qquad \eta >0. \label{pia}
\end{equation}%

Moreover, in the special case $\rho =1$, we obtain from (\ref{lp})
the Laplace transform of the $\alpha $-stable subordinator. As we
will see in the next section, both (\ref{lp}) and (\ref{pia}) are
completely monotone functions (with respect to $\eta _{1,}...,\eta
_{n}$ and $\eta ,$ respectively), under the conditions that
$\alpha,\rho \leq 1$ and $\alpha \rho \leq 1$, respectively. Thus
we will assume hereafter $\rho \in (0,1/\alpha ]$ when treating
the one-dimensional distribution, whereas we will restrict to the
case $\rho \in (0,1]$ in the $n$-dimensional case.

Moment generating functions expressed as power series of
incomplete gamma functions can be also found in \cite{STA}, where
they are used to generalize the gamma distribution.

The process $\mathcal{X}_{\alpha ,\rho }:=\{\mathcal{X}_{\alpha
,\rho }(t),t\geq 0\}$ defined in (\ref{lp}) displays a
self-similar property of order $1/\alpha $, for any $\rho $. On
the other hand, it is a L\'{e}vy process only in the stable case
(i.e. for $\rho =1$), since its distribution is not, in general,
infinitely divisible. The main feature of self-similar processes
is the invariance in distribution under suitable scaling of time
and space. Self-similarity is usually associated with long-range
dependence (see, for example, \cite{PIP}). The main tools used to
describe the dependence structure of a stochastic processes are
both the autocorrelation coefficient and the spectral analysis.
However, in the case of heavy-tailed distributions, where the
moments are not finite, we must resort to other tools, such as the
auto-codifference function. In particular, referring to the
definition given in \cite{SAM} and \cite{WYL3}, we are able to
interpret $\rho $ as a a measure of the local deviation from the
temporal dependence structure displayed in the standard $\alpha
$-stable case.

In order to derive the differential equation satisfied by the
transition function $h_{\alpha ,\rho }(x,t)$ of
$\mathcal{X}_{\alpha ,\rho },$ we distinguish the two cases
$0<\rho \leq 1$ and $1<\rho \leq 1/\alpha $: indeed, in the first
one, we prove that $h_{\alpha ,\rho }(x,t)$ satisfies an
integro-differential equation where a space-dependent
time-fractional derivative (of convolution-type) and a Caputo
space-fractional derivative appear. In the case where $\rho >1,$
the differential equation satisfied by the transition density is
obtained only in the special case $\rho =2$ and it
involves the so-called logarithmic differential operator introduced in \cite%
{BEG3} and defined in (\ref{sy2}) below.

We are able to write the explicit expression of $h_{\alpha ,\rho
}(x,t)$, in terms of $H$-functions, which for $\rho =1$ and $t=1,$
reduces to the stable law formula given in (2.28) of \cite{MAT}.
Finally we present some properties of the generalized stable r.v.
$\mathcal{X}_{\alpha ,\rho }$, such as the tails' behavior and the
moments. In particular, we prove that it possesses finite moments
of order $\delta <\alpha \rho $, which can be
written as%
\begin{equation}
\mathbb{E}\left( \mathcal{X}_{\alpha ,\rho }\right) ^{\delta }=\frac{c^{%
\frac{\delta }{\alpha }}}{\Gamma (\rho )}\frac{\Gamma (\rho -\frac{\delta }{%
\alpha })}{\Gamma (1-\delta )}.  \notag
\end{equation}

\quad

We now recall some well-known definitions and results which will
be used throughout the paper.

Let ${\LARGE H}_{p,q}^{m,n}$ denote the Fox's $H$-function defined
as (see
\cite{MAT} p.13):%
\begin{eqnarray}
&&{\LARGE H}_{p,q}^{m,n}\left[ \left. z\right\vert
\begin{array}{ccc}
(a_{1},A_{1}) & ... & (a_{p},A_{p}) \\
(b_{1},B_{1}) & ... & (b_{q},B_{q})%
\end{array}%
\right] \label{hh} \\
&& \doteq \frac{1}{2\pi i}\int_{\mathcal{L}}\frac{\left\{
\prod\limits_{j=1}^{m}\Gamma (b_{j}+B_{j}s)\right\} \left\{
\prod\limits_{j=1}^{n}\Gamma (1-a_{j}-A_{j}s)\right\}
z^{-s}ds}{\left\{ \prod\limits_{j=m+1}^{q}\Gamma
(1-b_{j}-B_{j}s)\right\} \left\{ \prod\limits_{j=n+1}^{p}\Gamma
(a_{j}+A_{j}s)\right\} },\notag
\end{eqnarray}%
with $z\neq 0,$ $m,n,p,q\in \mathbb{N}_{0},$ for $1\leq m\leq q$,
$0\leq
n\leq p$, $a_{j},b_{j}\in \mathbb{R},$ $A_{j},B_{j}\in \mathbb{R}^{+},$ for $%
i=1,...,p,$ $j=1,...,q$ and $\mathcal{L}$ is a contour such that
the
following condition is satisfied%
\begin{equation}
A_{l}(b_{j}+\alpha )\neq B_{j}(a_{l}-k-1),\qquad j=1,...,m,\text{ }l=1,...,n,%
\text{ }\alpha ,k=0,1,...  \label{1.6}
\end{equation}%

Let us recall the following particular case of the $H$-function
which we
will use later, i.e. the Mittag-Leffler function, defined as%
\begin{equation}
E_{\alpha ,\beta }(z):=\sum_{j=0}^{+\infty }\frac{z^{j}}{\Gamma
(\alpha j+\beta )}={\LARGE H}_{1,2}^{1,1}\left[ \left.
-z\right\vert
\begin{array}{cc}
(0,1) & \, \\
(0,1) & (1-\beta ,\alpha )%
\end{array}%
\right] \text{,}  \label{re}
\end{equation}%
for $\alpha >0$ (see formula (1.136) in \cite{MAT}). Moreover, we
recall the other special case of (\ref{hh}) given by the Meijer's
$G$-function, which
is obtained for $A_{i}=1,$ for any $i=1,...,p$ and for $B_{j}=1,$ for any $%
j=1,...,q$:%
\begin{equation}
G_{p,q}^{m,n}\left[ \left. z\right\vert
\begin{array}{c}
(a_{1},...a_{p}) \\
(b_{1},...,b_{q})%
\end{array}%
\right] :=H_{p,q}^{m,n}\left[ \left. z\right\vert
\begin{array}{ccc}
(a_{1},1) & ... & (a_{p},1) \\
(b_{1},1) & ... & (b_{q},1)%
\end{array}%
\right] .  \label{mei}
\end{equation}

In particular, we will use (\ref{mei}) in the case where $n=0$ and
$p=q=m,$
for which the following result holds (see properties 3 and 9 in \cite{KAR}):%
\begin{equation}
G_{p,p}^{p,0}\left[ \left. x\right\vert
\begin{array}{c}
(a_{1},...a_{p}) \\
(b_{1},...,b_{q})%
\end{array}%
\right] =0,\text{ \qquad }|x|>1,  \label{mei2}
\end{equation}%
while, under the assumption that $\sum_{j=1}^{p}\left(
t^{b_{j}}-t^{a_{j}}\right) \geq 0,$ it is%
\begin{equation}
G_{p,p}^{p,0}\left[ \left. x\right\vert
\begin{array}{c}
(a_{1},...a_{p}) \\
(b_{1},...,b_{q})%
\end{array}%
\right] \geq 0,\text{ \qquad }0<|x|<1.  \label{mei3}
\end{equation}

More details on the properties of the function (\ref{mei}),
together with its applications in fractional calculus, can be
found in \cite{KIR}. In particular, for $p=q=1$, the function
${\LARGE G}^{1,0}_{1,1}$ is proved to be the kernel of the
Erdelyi-Kober fractional integral (see also \cite{KAR2}).

We will use the definition and properties of completely monotone
functions, both in the univariate and multivariate cases; thus we
recall that, for Theorem 4.2.1 p.87 in \cite{BOC}, a function
$f(z_{1},...,z_{n})$ in an octant $0<z_{j}<\infty ,$ $j=1,...,n$
can be represented as
\begin{equation}
f(z_{1},...,z_{n})=\int_{0}^{+\infty }...\int_{0}^{+\infty
}e^{-z_{1}x_{1}...-z_{n}x_{n}}d\rho (x_{1},...x_{n}) \label{kir}
\end{equation}%
for a measure $\rho $ such that $\rho (A)\geq 0,$ for any $A,$ if
and only if it is $C^{\infty }$ and
\begin{equation}
(-1)^{k_{1}+...+k_{n}}\frac{\partial ^{k_{1}+...+k_{n}}}{\partial
z_{1}^{k_{1}}...\partial z_{n}^{k_{n}}}f(z_{1},...,z_{n})\geq 0,
\label{rs}
\end{equation}%
for all combinations $k_{1}\geq 0,...,k_{n}\geq 0$ (the
representation being unique). The measure $\rho $ is called
spectral distribution of $f$ and the function satisfying
(\ref{rs}) is said to be completely monotone (hereafter CM). By
the way, we note that the function (\ref{kir}) is linked to the
kernel of a particular case of the Obrechkoff transform (see, for
example, \cite{DIM}).

We also recall the definition of the convolution-type derivative
on the
positive half-axes, in the sense of Caputo (see \cite{TOA}, Def.2.4, for $%
b=0 $): let $\overline{\nu }(ds)$ be a non-negative L\'{e}vy
measure (i.e.
satisfying the condition $\int_{0}^{+\infty }(z\wedge 1)\overline{\nu }%
(dz)<\infty $) and let $\nu (s)=\int_{s}^{+\infty }\overline{\nu
}(dz)$ be its tail, under the assumption that it is absolutely
continuous. Let, moreover, $g:\mathbb{R}^{+}\rightarrow
\mathbb{R}^{+}$, be a Bernstein
function, i.e. with the following representation $g(x)=a+bx+\int_{0}^{+%
\infty }(1-e^{-sx})\overline{\nu }(ds),$ for $a,b\geq 0,$ then%
\begin{equation}
\mathcal{D}_{t}^{g}u(t):=\int_{0}^{t}\frac{d}{ds}u(t-s)\nu
(s)ds,\qquad t>0, \label{toa}
\end{equation}%
for an absolutely continuous function such that $|u(t)|\leq
Me^{\lambda t}$, for some $M,\lambda >0$ and for any $t.$

Convolution-type derivatives (or derivatives defined as integrals
with memory kernels) have been treated recently by many authors:
see, among the others, \cite{KOC}, \cite{KOL}, \cite{STAW},
\cite{GAJ}, \cite{TOA}.

The Laplace transform of $\mathcal{D}_{t}^{g}$ is given by
\begin{equation}
\int_{0}^{+\infty }e^{-\theta t}\mathcal{D}_{t}^{g}u(t)dt=g(\theta )%
\widetilde{u}(\theta )-\frac{g(\theta )}{\theta }u(0),\qquad \mathcal{R}%
(\theta )>\theta _{0},  \label{lapconv}
\end{equation}%
(see \cite{TOA}, Lemma 2.5). It is easy to check that, in the
trivial case where $g(\theta )=\theta $, the convolution-type
derivative coincides with the standard first-order derivative.
When $g(\theta )=\theta ^{\alpha }$, for $\alpha \in (0,1),$ this
Bernstein function coincides with the Laplace exponent of the
$\alpha $-stable subordinator $\mathcal{S}_{\alpha }(t)$
(distributed, for any $t,$ as a stable r.v. $\mathcal{S}_{\alpha
}(\sigma ,\beta ,\mu )$, with $\sigma =(\cos \left( \pi \alpha
/2\right) t)^{1/\alpha },$ $\beta =1$ and $\mu =0$), since
\begin{equation}
\mathbb{E}e^{-\theta \mathcal{S}_{\alpha }(t)}=\exp \left\{
-\theta ^{\alpha }t\right\} ,\qquad t,\theta >0.  \label{dd}
\end{equation}%

In this case, the L\'{e}vy measure reduces to $\overline{\nu
}(ds)=\alpha s^{-\alpha -1}/\Gamma (1-\alpha )$ and thus
(\ref{toa}) coincides with the fractional Caputo derivative of
order $\alpha $, i.e.

\begin{equation}
^{C}\mathcal{D}_{t}^{\alpha }u(t):=\frac{1}{\Gamma (1-\alpha )}\int_{0}^{t}%
\frac{d}{ds}u(s)(t-s)^{-\alpha }ds,\qquad t>0,\;\alpha \in (0,1).
\label{cap}
\end{equation}%

Finally, we will use the Riesz-Feller (RF) fractional derivative,
which is defined, by means of its Fourier transform. Let
$L^{c}(\mathcal{I})$ of
functions for which the Riemann improper integral on any open interval $%
\mathcal{I}$ absolutely converges (see \cite{MAR}). For any $f\in L^{c}(%
\mathcal{I})$, the RF fractional derivative is defined as

\begin{equation}
\mathcal{F}\left\{ \mathcal{D}_{x,\theta }^{\alpha }f(x);\xi
\right\} =-\psi _{\theta }^{\alpha }(\xi )\mathcal{F}\left\{
f(x);\xi \right\} ,\quad \alpha \in (0,2],\quad |\theta |\leq \min
\{\alpha ,2-\alpha \},  \label{uno}
\end{equation}%
with symbol
\begin{equation}
\psi _{\theta }^{\alpha }(\xi ):=|\xi |^{\alpha }e^{i\,sign(\xi
)\theta \pi /2},\qquad \mathcal{\xi \in }\mathbb{R},  \label{due}
\end{equation}%
(see \cite{KIL}, p.131). Since we are interested in the special
case where the support of the density is the positive half-line,
we can take $\xi >0$ and $\theta =-\alpha $ (which corresponds to
the stable law totally skewed to the right).

\section{Tempered relaxation function}\label{sec:2}

\setcounter{section}{2}
\setcounter{equation}{0}\setcounter{theorem}{0}

Under the assumption that $\rho \leq 1,$ we prove that the
(normalized) upper-incomplete Gamma function given in (\ref{trf})
is the solution of a
tempered relaxation equation, which extends the standard one, i.e. (\ref{rr}%
), by means of the tempered fractional derivative, in the Caputo
sense. The latter can be obtained from the definition given in
(\ref{toa}), by choosing
the Bernstein function $g(x)=(\lambda +x)^{\rho }-\lambda ^{\rho },$ for $%
\lambda ,x\in \mathbb{R}^{+}$ and the L\'{e}vy measure $\overline{\nu }(ds)=%
\frac{\rho e^{-\lambda s}s^{-\rho -1}}{\Gamma (1-\rho )}ds$. Thus the tail L%
\'{e}vy measure is taken equal to $\nu (ds)=\frac{\rho \lambda
^{\rho
}\Gamma (-\rho ;\lambda s)}{\Gamma (1-\rho )}ds$ (see \cite{TOA} and \cite%
{STAW}, for details). For $\rho =1$, we have that $g(x)=x$ and
thus the tempered derivative reduces to the first-order
derivative.

\begin{lemma}
\label{Lemrel} Let $\mathcal{D}_{t}^{\lambda ,\rho }$ be defined,
for any absolutely continuous function
$u:\mathbb{R}^{+}\rightarrow \mathbb{R}^{+}$,
such that $|u(t)|\leq ce^{kt},$ for some $c,k>0$ and for any $t\geq 0,$ as%
\begin{equation}
\mathcal{D}_{t}^{\lambda ,\rho }u(t)=\frac{\rho \lambda ^{\rho
}}{\Gamma (1-\rho )}\int_{0}^{t}\frac{d}{dt}u(t-s)\Gamma (-\rho
;\lambda s)ds,\qquad \lambda >0,\text{ }\rho \in (0,1],
\label{temp}
\end{equation}%
then the solution of the following tempered relaxation equation%
\begin{equation}
\mathcal{D}_{t}^{\lambda ,\rho }u(t)=-\lambda ^{\rho }u(t),
\label{temp2}
\end{equation}%
with initial condition $u(0)=1$, is given by%
\begin{equation}
\varphi _{\lambda ,\rho }(t)=\frac{\Gamma (\rho ;\lambda
t)}{\Gamma (\rho )}. \label{fl}
\end{equation}
\end{lemma}

\begin{proof}
The Laplace transform of the l.h.s. of (\ref{temp2}) can be obtained from (%
\ref{lapconv}), by choosing $g(x)=(\lambda +x)^{\rho }-\lambda
^{\rho }$ and
considering the initial condition:%
\begin{equation}
\int_{0}^{+\infty }e^{-\theta t}\mathcal{D}_{t}^{\lambda ,\rho
}u(t)dt=\left[
(\lambda +\theta )^{\rho }-\lambda ^{\rho }\right] \widetilde{u}(\theta )-%
\frac{(\lambda +\theta )^{\rho }-\lambda ^{\rho }}{\theta }.
\label{templap}
\end{equation}%

Then, by taking into account the r.h.s. of (\ref{temp2}), we get%
\begin{equation}
\widetilde{u}(\theta )=\frac{(\lambda +\theta )^{\rho }-\lambda ^{\rho }}{%
\theta (\lambda +\theta )^{\rho }}  \label{temp3}
\end{equation}%
which coincides with the Laplace transform of (\ref{fl}), since%
\begin{eqnarray*}
\frac{1}{\Gamma (\rho )}\int_{0}^{+\infty }e^{-\theta t}\Gamma
\left( \rho ;\lambda t\right) dt &=&\frac{1}{\Gamma (\rho
)}\int_{0}^{+\infty
}e^{-\theta t}\int_{\lambda t}^{+\infty }e^{-w}w^{\rho -1}dwdt \\
&=&\frac{\lambda ^{\rho }}{\Gamma (\rho )}\int_{0}^{+\infty
}e^{-\lambda
w}w^{\rho -1}\int_{0}^{w}e^{-\theta t}dtdw \\
&=&\frac{\lambda ^{\rho }}{\theta \Gamma (\rho )}\int_{0}^{+\infty
}e^{-\lambda w}(1-e^{-\theta w})w^{\rho -1}dw.
\end{eqnarray*}
\end{proof}

We now derive some properties of $\varphi _{\lambda ,\rho }(\cdot
)$, which are usually required to a relaxation function, and
obtain its spectral distribution in terms of H-function. The
latter will be proved to coincide with the density of the process
$\mathcal{X}_{\alpha ,\rho }$ defined in the next section (in the
special case $\alpha =1$).

\begin{theorem}\label{Th1}
The tempered relaxation function $\varphi _{\lambda ,\rho }(\cdot
)$, defined in (\ref{fl}) is $\mathcal{C}^{\infty }$, completely
monotone and locally integrable. The spectral distribution of
$\varphi _{\lambda ,\rho
}(\cdot )$ is given by%
\begin{equation}
K_{\lambda }(z)=\frac{\lambda 1_{z>\lambda }}{\Gamma (\rho )}{\LARGE G}%
_{1,1}^{1,0}\left[ \left. \frac{\lambda }{z}\right\vert
\begin{array}{c}
2 \\
1+\rho%
\end{array}%
\right] ,\qquad \rho \in (0,1),  \label{k}
\end{equation}%
where ${\LARGE G}_{p,q}^{m,n}[\cdot ]$ is the Meijer's
$G$-function defined in (\ref{mei}).
\end{theorem}

\begin{proof}
It is well-known that the upper incomplete gamma function is $\mathcal{C}%
^{\infty }$ and thus the same holds true for $\varphi _{\lambda
,\rho }(\cdot ).$ Moreover, we can derive complete monotonicity by
repeatedly
differentiating (\ref{fl}) and showing that $(-1)^{k}\frac{d^{k}}{dt^{k}}%
\varphi _{\lambda ,\rho }(t)\geq 0,$ for any $k\in \mathbb{N}$ and $t\geq 0$:%
\begin{eqnarray*}
-\frac{d}{dt}\varphi _{\lambda ,\rho }(t) &=&\lambda ^{\rho
}e^{-\lambda
t}t^{\rho -1}\geq 0 \\
\frac{d^{2}}{dt^{2}}\varphi _{\lambda ,\rho }(t) &=&\lambda ^{\rho
+1}e^{-\lambda t}t^{\rho -1}-(\rho -1)\lambda ^{\rho }e^{-\lambda
t}t^{\rho
-2}\geq 0 \\
-\frac{d^{3}}{dt^{3}}\varphi _{\lambda ,\rho }(t) &=&\lambda
^{\rho +2}e^{-\lambda t}t^{\rho -1}-2(\rho -1)\lambda ^{\rho
+1}e^{-\lambda t}t^{\rho -2}+\\
&+&(\rho -1)(\rho -2)\lambda ^{\rho }e^{-\lambda t}t^{\rho -3}\geq
0
\end{eqnarray*}%
and so on. The sign of the $k$-derivative depends only on $\rho -j+1,$ for $%
j=2,...,k$, which is non-positive, for any $k\in \mathbb{N}$ and
for $\rho \leq 1.$ As far as the integrability, we first derive an
alternative expression of (\ref{fl}) in terms of Mittag-Leffler
function: recalling formula (3.7) p.316 in \cite{MIL} and formula
(4.2.8) in \cite{GOR}, we can
write%
\begin{eqnarray}
\varphi _{\lambda ,\rho }(t) &=&1-\frac{\int_{0}^{\lambda
t}e^{-w}w^{\rho
-1}dw}{\Gamma \left( \rho \right) }  \label{tt} \\
&=&1-e^{-\lambda t}\lambda ^{\rho }\mathcal{E}_{t}(\rho ,\lambda
)=1-e^{-\lambda t}(\lambda t)^{\rho }E_{1,\rho +1}(\lambda t),
\notag
\end{eqnarray}%
where $\mathcal{E}_{x}(\nu ,\lambda )=\frac{e^{\lambda x}}{\Gamma (\nu )}%
\int_{0}^{x}w^{\nu -1}e^{-w}dw$ is the so-called Miller-Ross
function.

By applying formula (6.5.31) in \cite{WYL}, we easily obtain the
following
limiting behavior of (\ref{fl})%
\begin{equation}
\varphi _{\lambda ,\rho }(t)\simeq \frac{e^{-\lambda t}\lambda
^{\rho -1}t^{\rho -1}}{\Gamma (\rho )},\qquad t\rightarrow +\infty
.  \label{ttt}
\end{equation}%

Since $\varphi _{\lambda ,\rho }(0)=1$ and the function $\varphi
_{\lambda ,\rho }(\cdot )$. is bounded by one on any compact set,
(\ref{ttt}) is
enough to prove its integrability. Finally, we check that%
\begin{equation}
\varphi _{\lambda ,\rho }(t)=\int_{0}^{+\infty }e^{-tz}K_{\lambda
}(z)dz, \label{kk}
\end{equation}%
for $K(z)$ given in (\ref{k}). Indeed we can write (\ref{k}), by (\ref{mei}%
), in terms of a $H$-functions, as%
\begin{eqnarray*}
K_{\lambda }(z) &=&\frac{\lambda }{\Gamma (\rho )}{\LARGE H}_{1,1}^{1,0}%
\left[ \left. \frac{\lambda }{z}\right\vert
\begin{array}{c}
(2,1) \\
(1+\rho ,1)%
\end{array}%
\right] \\
&=&[\text{ by(1.60) in \cite{MAT}}]=\frac{\lambda }{\Gamma (\rho )}{\LARGE H}%
_{1,1}^{0,1}\left[ \left. \frac{z}{\lambda }\right\vert
\begin{array}{c}
(-\rho ,1) \\
(-1,1)%
\end{array}%
\right] ,
\end{eqnarray*}%
so that we can apply formula (2.19) in \cite{MAT} and get%
\begin{eqnarray*}
&&\frac{1}{\lambda \Gamma (\rho )}\int_{0}^{+\infty }e^{-tz}{\LARGE H}%
_{1,1}^{0,1}\left[ \left. \frac{z}{\lambda }\right\vert
\begin{array}{c}
(-\rho ,1) \\
(-1,1)%
\end{array}%
\right] dz \\
&=&\frac{1}{\Gamma (\rho )\lambda t}{\LARGE H}_{2,1}^{0,2}\left[
\left. \frac{1}{\lambda t}\right\vert
\begin{array}{c}
(0,1)(-\rho ,1) \\
(-1,1)%
\end{array}%
\right] \\
&=&[\text{by (1.58) and (1.60) in \cite{MAT}}]=\frac{1}{\Gamma (\rho )}%
{\LARGE H}_{1,2}^{2,0}\left[ \left. \lambda t\right\vert
\begin{array}{c}
(1,1) \\
(0,1)(\rho ,1)%
\end{array}%
\right] \\
&=&\frac{1}{\Gamma (\rho )}\frac{1}{2\pi i}\int_{\mathcal{L}}(\lambda t)^{-s}%
\frac{\Gamma (s)\Gamma (\rho +s)}{\Gamma (1+s)}ds=\varphi
_{\lambda ,\rho }(t),
\end{eqnarray*}%
by means of the Mellin-Barnes integral form of the incomplete gamma function:%
\begin{equation}
\Gamma (\rho ;x)=\frac{1}{2\pi i}\int_{\mathcal{L}}\frac{\Gamma
(s+\rho )x^{-s}}{\Gamma (\rho )s}ds,  \label{MB}
\end{equation}%
where the contour $\mathcal{L}$ avoids all the poles of the gamma
functions
(see \cite{PAR} for more details). Note that the indicator function in (\ref%
{k}) comes from (\ref{mei2}).
\end{proof}

\begin{remark}
We note that, as $t\rightarrow +\infty ,$ the function $\varphi
_{\lambda ,\rho }(\cdot )$ decays exponentially fast and thus
faster than in the fractional case. On the other hand, for small
values of $t$, the tempered relaxation function exhibits the same
fast decay of the fractional relaxation (see \cite{MAI}): indeed
we can obtain the behavior of $\varphi
_{\lambda ,\rho }(\cdot )$, for $t\rightarrow 0^{+}$, by considering (\ref%
{tt}) together with%
\begin{equation*}
E_{\alpha ,\beta }(z)\simeq \frac{1}{\Gamma (\beta
)}+\frac{z}{\Gamma (\beta +\alpha )},\qquad z\rightarrow 0^{+}
\end{equation*}%
(see \cite{GOR}, p.64): thus we have that
\begin{equation}
\varphi _{\lambda ,\rho }(t)\simeq 1-\frac{e^{-\lambda t}\lambda
^{\rho
}t^{\rho }}{\Gamma (\rho +1)}\simeq 1-\frac{t^{\rho }}{\Gamma (\rho +1)}%
,\qquad t\rightarrow 0^{+}.
\end{equation}

We thus obtain asymptotic behaviors similar to those of the
relaxation function introduced in \cite{STAW} (by time-changing
with inverse tempered subordinators), but, in this case, we have
also a closed expression of the solution, in terms of well-known
functions.
\end{remark}

\begin{remark}
As far as the spectral distribution $K_{\lambda }(z)$, $z\geq 0,$ given in (%
\ref{k}), we can check, by (\ref{mei3}), that it is a proper
probability distribution: indeed it is non-negative, since
$t^{1+\rho }-t^{2}\geq 0$, for $\rho <1$ and $t\in (0,1).$
Furthermore, $\int_{0}^{+\infty }K_{\lambda }(z)dz=1$, by using
(\ref{kk}) and considering that $\varphi _{\lambda ,\rho }(0)=1.$
From (\ref{k}) we can draw the conclusion that the process
governed by $\varphi _{\lambda ,\rho }(t)$ can be expressed in
terms of a continuous distribution of standard (exponential)
relaxation processes with frequencies on the range $(\lambda
,+\infty )$ instead of the whole real positive
semi-axes. In the special case $\rho =1/2$, it is well-known that%
\begin{equation*}
{\LARGE G}_{1,1}^{0,1}\left[ \left. x\right\vert
\begin{array}{c}
-\frac{1}{2} \\
-1%
\end{array}%
\right] =\frac{1_{x>1}}{\sqrt{\pi }x\sqrt{x-1}},
\end{equation*}%
(see (1.143) in \cite{MAT}), so that we get
\begin{equation}
K_{\lambda }(z)=\frac{\lambda ^{5/2}1_{z>\lambda }}{\pi z\sqrt{z-\lambda }}%
{\LARGE .}
\end{equation}
\end{remark}

\begin{remark}
Let $S_{\lambda ,\rho }(t),t\geq 0,$ be the tempered subordinator, for $%
\lambda \geq 0$ and $\rho \in (0,1],$ with transition density
$h_{\lambda ,\rho }(z,t)$. Let moreover $L_{\lambda ,\rho
}(t),t\geq 0,$ be its inverse, i.e. $L_{\lambda ,\rho }(t):=\inf
\{s\geq 0,\,S_{\lambda ,\rho }(s)>t\}$ (see \cite{BEG},
\cite{GAJ2}, \cite{MEE} and \cite{WYL2}, for details). As a
consequence of Lemma \ref{Lemrel} we can derive the following
relationship between the tempered relaxation and the density of
$L_{\lambda ,\rho },$
i.e. $l_{\lambda ,\rho }(x,t):=P\{L_{\lambda ,\rho }(t)\in dx\}$:%
\begin{equation}
\varphi _{\lambda ,\rho }(t)=\frac{\Gamma \left( \rho ;\lambda t\right) }{%
\Gamma (\rho )}=\int_{0}^{+\infty }e^{-\lambda ^{\rho
}z}l_{\lambda ,\rho }(z,t)dz.  \label{inv}
\end{equation}%

The previous integral expression of the incomplete gamma function
can be obtained, by applying (19) of \cite{WYL}. Let us denote by
$h_{\rho }(u,z)$
the density of the $\rho $-stable subordinator (i.e. for $\lambda =0$), then%
\begin{eqnarray*}
&&\int_{0}^{+\infty }e^{-\lambda ^{\rho }z}l_{\lambda ,\rho }(z,t)dz \\
&=&-\int_{0}^{+\infty }e^{-\lambda ^{\rho }z}\frac{\partial }{\partial z}%
\int_{0}^{t}h_{\lambda ,\rho }(u,z)dudz \\
&=&-\lambda ^{\rho }\int_{0}^{t}e^{-\lambda u}\int_{0}^{+\infty
}e^{-\lambda ^{\rho }z+\lambda ^{\rho }z}h_{\rho
}(u,z)dzdu+\int_{0}^{t}h_{\lambda ,\rho
}(u,0)du \\
&=&-\lambda ^{\rho }\int_{0}^{t}e^{-\lambda u}\int_{0}^{+\infty
}h_{\rho }(u,z)dzdu+1,
\end{eqnarray*}%
for $t>0,$ since $h_{\lambda ,\rho }(u,0)=\delta (u).$ We now notice that $%
\int_{0}^{+\infty }h_{\rho }(u,z)dz=u^{\rho -1}/\Gamma \left( \rho
\right) $ (as can be checked by the Laplace transform) and the
result (\ref{inv}) easily follows.

The relationship (\ref{inv}) can be also checked by considering
that both sides satisfy (\ref{temp2}): by taking the tempered
derivative of the r.h.s.
of (\ref{inv}):%
\begin{eqnarray*}
&&\mathcal{D}_{t}^{\lambda ,\rho }\left[ \Gamma \left( \rho
\right) \int_{0}^{+\infty }e^{-\lambda ^{\rho }z}l_{\lambda ,\rho
}(z,t)dz\right]
=\Gamma \left( \rho \right) \int_{0}^{+\infty }e^{-\lambda ^{\rho }z}%
\mathcal{D}_{t}^{\lambda ,\rho }l_{\lambda ,\rho }(z,t)dz \\
&=&[\text{by (6.13) and (2.33) in \cite{TOA}}] \\
&=&-\Gamma \left( \rho \right) \int_{0}^{+\infty }e^{-\lambda ^{\rho }z}%
\frac{\partial }{\partial z}l_{\lambda ,\rho }(z,t)dz-\Gamma
\left( \rho \right) \nu (t)\int_{0}^{+\infty }e^{-\lambda ^{\rho
}z}l_{\lambda ,\rho
}(z,0)dz \\
&=&-\left[ \Gamma \left( \rho \right) e^{-\lambda ^{\rho
}z}l_{\lambda ,\rho }(z,t)\right] _{z=0}^{+\infty }-\Gamma \left(
\rho \right) \nu(t)-%
\Gamma \left( \rho \right)\lambda ^{\rho } \int_{0}^{+\infty
}e^{-\lambda
^{\rho }z}l_{\lambda ,\rho }(z,t)dz \\
&=&-\Gamma \left( \rho \right) \lambda ^{\rho }\int_{0}^{+\infty
}e^{-\lambda ^{\rho }z}l_{\lambda ,\rho }(z,t)dz,
\end{eqnarray*}%
where we have considered the initial conditions $l_{\lambda ,\rho }(0,t)=%
\frac{\rho \lambda ^{\rho }\Gamma (-\rho ,\lambda t)}{\Gamma
(1-\rho )}$ and $l_{\lambda ,\rho }(x,0)=\delta (x)$ (see
\cite{TOA}, for details).
\end{remark}

\section{Generalized stable process}\label{sec:3}

\setcounter{section}{3}
\setcounter{equation}{0}\setcounter{theorem}{0}

We recall that the Laplace transform of the $\alpha $-stable
subordinator,
i.e. $\Phi (\eta ):=\mathbb{E}e^{-\eta S_{\alpha }(t)}=e^{-\eta ^{\alpha }t}$%
, $t\geq 0,$ satisfies the standard relaxation equation (\ref{rr}), with $%
\lambda =\eta ^{\alpha }.$ We then apply the previous results, in
order to define a new stochastic process, which generalizes the
$\alpha $-stable subordinator.

We first prove that, under some assumptions, the function
\begin{equation}
\Phi (\eta _{1},...\eta _{n};t_{1},...t_{n})=\frac{\Gamma \left(
\rho ;\sum\limits_{k=1}^{n}\Xi _{k,n}^{\alpha}\left(
t_{k}-t_{k-1}\right) \right) }{\Gamma \left( \rho \right) },\qquad
\eta _{1},...\eta _{n}\geq 0, n \in \mathbb{N},  \label{chi}
\end{equation}%
where $\Xi _{k,n}^{\alpha}:=\left(\sum_{i=k}^{n}\eta
_{i}\right)^{\alpha}$, can be uniquely
represented as a Laplace transform of a probability measure in $\mathbb{R}%
^{n}.$

\begin{lemma}
\label{LemCM} Let $0<t_{1}<....<t_{n},$ $\alpha \in (0,1)$ and
$\rho \in (0,1].$ The function $\Phi (\eta _{1},...\eta
_{n};t_{1},...t_{n})$ in (\ref{chi}) is $C^{\infty },$ is such
that $\Phi
(0,..,0;t_{1},...t_{n})=1 $ and is completely monotone (CM) with respect to $%
\eta _{1},...\eta _{n},$ for any choice of $t_{1}<...<t_{n}$. Thus
the
following integral (unique) representation holds%
\[
\Phi (\eta _{1},...\eta _{n};t_{1},...t_{n})=\int_{0}^{+\infty
}...\int_{0}^{+\infty }e^{-\eta _{1}x_{1}...-\eta _{n}x_{n}}\pi
(dx_{1},...,dx_{n};t_{1},...t_{n}),
\]%
for a probability measure $\pi (\cdot ,...,\cdot
;t_{1},...t_{n})$.
\end{lemma}

\begin{proof}
We first check that the one-dimensional function $\Phi (\eta ,t)=\mathbb{E}%
e^{-\eta \mathcal{X}_{\alpha ,\rho }(t)}$ given in (\ref{mlt}) is
$C^{\infty }$, since the incomplete gamma function is analytic
w.r.t. both its arguments. Moreover, it is such that $\Phi
(0,t)=1$ and is a CM function, i.e is absolutely differentiable
and $(-1)^{k}\frac{d^{k}}{d\eta ^{k}}\Phi (\eta ,t)\geq 0$, for
$k=0,1,....$ Indeed, $\eta ^{\alpha }$ is a Bernstein function and
we have proved in Theorem \ref{Th1} that $\Gamma (\rho ;\cdot )$
is CM. Then it is well-known that the composition of a CM and a
Bernstein function is again CM. Thus the statement is proved for
$n=1.$ Moreover, we can check directly that this is true if and
only if $\alpha \rho <1$:
\begin{eqnarray*}
&&\frac{d}{d\eta }\Phi (\eta ,t)=-\alpha t^{\rho }\eta ^{\alpha
\rho-1}e^{-\eta ^{\alpha }t}\leq 0 \\
&&\frac{d^{2}}{d\eta ^{2}}\Phi (\eta ,t)=-\alpha (\alpha \rho
-1)t^{\rho }\eta ^{\alpha \rho -2}e^{-\eta ^{\alpha }t}+\alpha
^{2}t^{\rho +1}\eta
^{\alpha \rho +\alpha -2}e^{-\eta ^{\alpha }t}\geq 0 \\
&&\frac{d^{3}}{d\eta ^{3}}\Phi (\eta ,t)=-\alpha (\alpha \rho
-1)\left( \alpha \rho -2\right) t^{\rho }\eta ^{\alpha \rho
-3}e^{-\eta ^{\alpha }t}+(\alpha \rho -1)\alpha ^{2}t^{\rho
+1}\eta ^{\alpha \rho +\alpha
-3}e^{-\eta ^{\alpha }t}+ \\
&&+\alpha ^{2}t^{\rho +1}(\alpha \rho +\alpha -2)\eta ^{\alpha
\rho +\alpha -3}e^{-\eta ^{\alpha }t}-\alpha ^{3}t^{\rho +2}\eta
^{\alpha \rho +2\alpha -3}e^{-\eta ^{\alpha }t}\leq 0
\end{eqnarray*}%
and so on. The sign of the $k$-th derivatives depends only on
$\alpha \rho -j+1,$ for $j=2,...k,$ and $\alpha \rho +(j-2)\alpha
-j+1,$ for $j=3,...k,$ which are all non-positive under the unique
condition $\alpha \rho <1$.

For $n>1,$ we start by proving that the condition (\ref{rs}) holds, for $%
n=2, $ i.e. that
\begin{eqnarray*}
&&(-1)^{k_{1}+k_{2}}\frac{\partial ^{k_{1}+k_{2}}}{\partial \eta
_{1}^{k_{1}}\partial \eta _{2}^{k_{2}}}\Phi (\eta _{1},\eta
_{2};t_{1},t_{2})
\\
&=&(-1)^{k_{1}+k_{2}}\frac{\partial ^{k_{1}+k_{2}}}{\partial \eta
_{1}^{k_{1}}\partial \eta _{2}^{k_{2}}}\frac{\Gamma \left( \rho
;(\eta
_{1}+\eta _{2})^{\alpha }t_{1}+\eta _{2}^{\alpha }(t_{2}-t_{1})\right) }{%
\Gamma \left( \rho \right) }\geq 0\
\end{eqnarray*}

Let us denote $A_{1,2}\doteq (\eta _{1}+\eta _{2})^{\alpha
}t_{1}+\eta
_{2}^{\alpha }(t_{2}-t_{1})$; thus we can write%
\begin{eqnarray*}
&&(-1)^{k_{1}+k_{2}}\frac{\partial ^{k_{1}+k_{2}}}{\partial \eta
_{1}^{k_{1}}\partial \eta _{2}^{k_{2}}}\Phi (\eta _{1},\eta
_{2};t_{1},t_{2})
\\
&=&\frac{(-1)^{k_{1}+k_{2}}}{\Gamma \left( \rho \right)
}\frac{\partial ^{k_{2}}}{\partial \eta
_{2}^{k_{2}}}\frac{\partial ^{k_{1}-1}}{\partial \eta
_{1}^{k_{1}-1}}\frac{\partial }{\partial \eta _{1}}\Gamma \left(
\rho
;A_{1,2}\right) \\
&=&\frac{(-1)^{k_{1}+k_{2}+1}}{\Gamma \left( \rho \right)
}\frac{\partial ^{k_{2}}}{\partial \eta
_{2}^{k_{2}}}\frac{\partial ^{k_{1}-1}}{\partial
\eta _{1}^{k_{1}-1}}\left[ e^{-A_{1,2}}A_{1,2}^{\rho -1}\frac{\partial }{%
\partial \eta _{1}}A_{1,2}\right] \\
&=&[\text{by the general Leibniz formula}] \\
&=&\frac{(-1)^{k_{2}}}{\Gamma \left( \rho \right) }\frac{\partial ^{k_{2}}}{%
\partial \eta _{2}^{k_{2}}}{\LARGE [}\sum_{j_{1}+j_{2}+j_{3}=k_{1}-1}\binom{%
k_{1}-1}{j_{1},j_{2},j_{3}}(-1)^{j_{1}}\frac{\partial
^{j_{1}}}{\partial
\eta _{1}^{j_{1}}}e^{-A_{1,2}}\cdot \\
&&\cdot (-1)^{j_{2}}\frac{\partial ^{j_{2}}}{\partial \eta _{1}^{j_{2}}}%
A_{1,2}^{\rho -1}(-1)^{j_{3}+2}\frac{\partial ^{j_{3}+1}}{\partial
\eta
_{1}^{j_{3}+1}}A_{1,2}{\LARGE ]} \\
&=&\frac{1}{\Gamma \left( \rho \right) }\sum_{l_{1}+l_{2}+l_{3}=k_{2}}\binom{%
k_{2}}{l_{1},l_{2},l_{3}}\sum_{j_{1}+j_{2}+j_{3}=k_{1}-1}\binom{k_{1}-1}{%
j_{1},j_{2},j_{3}}\left[ (-1)^{l_{1}+j_{1}}\frac{\partial ^{l_{1}+j_{1}}}{%
\partial \eta _{2}^{l_{1}}\partial \eta _{1}^{j_{1}}}e^{-A_{1,2}}\right]
\cdot \\
&&\cdot \left[ (-1)^{l_{2}+j_{2}}\frac{\partial
^{l_{2}+j_{2}}}{\partial \eta _{2}^{l_{2}}\partial \eta
_{1}^{j_{2}}}A_{1,2}^{\rho -1}\right] \left[
(-1)^{l_{3}+j_{3}}\frac{\partial ^{l_{3}+j_{3}+1}}{\partial \eta
_{2}^{l_{3}}\partial \eta _{1}^{j_{3}+1}}A_{1,2}\right] .
\end{eqnarray*}

Under the assumptions that $\alpha \in (0,1) $ and $\rho \in (0,1]$ and recalling that $%
t_{1},t_{2}-t_{1}\geq 0,$ it is easy to show that $A_{1,2}$ is a
Bernstein
function w.r.t. $\eta _{1}$ and $\eta _{2}$. On the other hand, $%
A_{1,2}^{\rho -1}$ and $e^{-A_{1,2}}$ are both CM, being the
composition of CM and Bernstein functions (see Theorem 3.7 in
\cite{SCH}, p.27). Moreover it can be checked that, in the last
step, the sign of the partial derivatives depends only on the
order, regardless of the variables involved; thus the functions
inside square brackets are all non-negative. This procedure can be
successively applied in order to prove that (\ref{rs}) is true for
any integer $n\geq 3.$
\end{proof}

\begin{definition}
\label{DefLT} We define the process $\mathcal{X}_{\alpha ,\rho
}:=\left\{ \mathcal{X}_{\alpha ,\rho }(t),t\geq 0\right\} $, under
the assumption that, for $t_{0}=0$, $\mathcal{X}_{\alpha ,\rho
}(t_{0})=0,$ almost surely, by the $n$-times Laplace transform of
its finite dimensional distributions given in (\ref{chi}), i.e.
\begin{equation*}
\mathbb{E}e^{-\sum_{k=1}^{n}\eta _{k}\mathcal{X}_{\alpha ,\rho
}(t_{k})}:=\frac{\Gamma \left( \rho ;\sum\limits_{k=1}^{n}\Xi
_{k,n}^{\alpha}\left( t_{k}-t_{k-1}\right) \right) }{\Gamma \left(
\rho \right) },\qquad \eta _{1},...\eta _{n}\geq 0, n \in
\mathbb{N},
\end{equation*}
for $\Xi _{k,n}^{\alpha}:=\left(\sum_{i=k}^{n}\eta
_{i}\right)^{\alpha}$ and $\alpha \in (0,1)$, $\rho \in (0,1]$.
\end{definition}

\begin{remark}
It is easy to check that $\mathcal{X}_{\alpha ,\rho }$ displays
the self-similar property of order $1/\alpha $, for any $\rho $:
indeed we have the following equality of all the
finite-dimensional distributions $\left\{
a^{\frac{1}{\alpha }}\mathcal{X}_{\alpha ,\rho }(t),t\geq 0\right\} \overset{%
d}{=}\left\{ \mathcal{X}_{\alpha ,\rho }(at),t\geq 0\right\} $, for any $%
a\geq 0$.

On the other hand, it is a L\'{e}vy process only for $\rho =1$,
when its distribution becomes infinitely divisible. Moreover, only
in this special case the process has stationary and independent
increments, since formula (\ref{chi}) can be rewritten as
\begin{eqnarray*}
\Phi (\eta _{1},...\eta _{n};t_{1},...t_{n}) &=&\exp \left\{
\sum\limits_{k=1}^{n}\Xi _{k,n}^{\alpha}\left(
t_{k}-t_{k-1}\right) \right\}
=\prod\limits_{k=1}^{n}e^{-\Xi _{k,n}^{\alpha}\left( t_{k}-t_{k-1}\right) } \\
&=&\prod\limits_{k=1}^{n}\mathbb{E}e^{-\Xi
_{k,n}^{\alpha}\mathcal{X}_{\alpha ,1}\left( t_{k}-t_{k-1}\right)
}=\prod\limits_{k=1}^{n}\mathbb{E}e^{-\Xi
_{k,n}^{\alpha}\left[ \mathcal{X}_{\alpha ,1}(t_{k})-\mathcal{X}_{\alpha ,1}(t_{k-1})%
\right] }. \label{iis}
\end{eqnarray*}
It is evident that, for $\rho =1$, the increments have an
$\alpha$-stable distribution (see \cite{SAM}, p.113) and thus
$\mathcal{X}_{\alpha ,1 }$ reduces to the stable process.

\end{remark}

We now consider the one-dimensional distribution of the process
and we prove
that it can be written in terms of the Fox's $H$-function defined in (\ref%
{hh}).

\begin{theorem}\label{Th8}
The transition density $h_{\alpha }^{\rho }(x,t)$ of
$\mathcal{X}_{\alpha
,\rho }$, for any $\rho \in (0,1/\alpha ],$ is given by%
\begin{equation}
h_{\alpha }^{\rho }(x,t)=\frac{\alpha }{x\Gamma (\rho )}{\LARGE H}%
_{1,1}^{0,1}\left[ \left. \frac{x^{\alpha }}{t}\right\vert
\begin{array}{c}
(1-\rho ,1) \\
(0,\alpha )%
\end{array}%
\right] ,\qquad x>0,\text{ \ }t\geq 0.  \label{sta}
\end{equation}
\end{theorem}

\begin{proof}
We invert the one-dimensional Laplace transform of the process $\mathcal{X}%
_{\alpha ,\rho },i.e.$
\begin{equation}
\mathbb{E}e^{-\eta \mathcal{X}_{\alpha ,\rho }(t)}=\frac{\Gamma
\left( \rho ;\eta ^{\alpha }t\right) }{\Gamma \left( \rho \right)
},\qquad \eta >0, \label{mlt}
\end{equation}%
by using (\ref{MB}), as follows%
\begin{eqnarray}
h_{\alpha }^{\rho }(x,t) &=&\frac{1}{2\pi i}\int_{\gamma -i\infty
}^{\gamma +i\infty }\frac{\Gamma (s+\rho )(tx^{-\alpha
})^{-s}}{x\Gamma (\rho )\Gamma (\alpha s)s}ds=\frac{\alpha }{2\pi
i}\int_{\gamma -i\infty }^{\gamma +i\infty }\frac{\Gamma (s+\rho
)(tx^{-\alpha })^{-s}}{x\Gamma (\rho )\Gamma
(\alpha s+1)}ds  \notag \\
&=&\frac{\alpha }{x\Gamma (\rho )}{\LARGE H}_{1,1}^{1,0}\left[
\left. tx^{-\alpha }\right\vert
\begin{array}{c}
(1,\alpha ) \\
(\rho ,1)%
\end{array}%
\right] =\frac{\alpha }{x\Gamma (\rho )}{\LARGE
H}_{1,1}^{0,1}\left[ \left. \frac{x^{\alpha }}{t}\right\vert
\begin{array}{c}
(1-\rho ,1) \\
(0,\alpha )%
\end{array}%
\right] , \label{den}
\end{eqnarray}%
which exists for any $x\neq 0$, since $\alpha -1<0$ (see Theorem 1.1 in \cite%
{MAT}).
\end{proof}

\begin{remark}
We can check that (\ref{sta}) is a proper density function, for
any $t$, by
applying the Mellin transform and recalling formula (2.8) in \cite{MAT}:%
\begin{eqnarray*}
\int_{0}^{+\infty }h_{\alpha }^{\rho }(x,t)dx &=&\frac{\alpha
}{\Gamma (\rho
)}\int_{0}^{+\infty }\frac{1}{x}{\LARGE H}_{1,1}^{0,1}\left[ \left. \frac{%
x^{\alpha }}{t}\right\vert
\begin{array}{c}
(1-\rho ,1) \\
(0,\alpha )%
\end{array}%
\right] dx \\
&=&\frac{1}{\Gamma (\rho )}\int_{0}^{+\infty }\frac{1}{y}{\LARGE H}%
_{1,1}^{0,1}\left[ \left. \frac{y}{t}\right\vert
\begin{array}{c}
(1-\rho ,1) \\
(0,\alpha )%
\end{array}%
\right] dy=\frac{1}{\Gamma (\rho )}\Gamma (\rho )=1.
\end{eqnarray*}
\end{remark}

The $H$-function distribution, which includes, among others, the
gamma, beta, Weibull, chi-square, exponential and the half-normal
distributions as particular cases, has been introduced by
\cite{CAR} and studied in many other papers (see, for example,
\cite{SPR}, \cite{VEL}).

For $\rho =1,$ formula (\ref{sta}) reduces to%
\begin{eqnarray*}
h_{\alpha }^{1}(x,t) &=&\alpha x^{-1}{\LARGE H}_{1,1}^{0,1}\left[
\left. \frac{x^{\alpha }}{t}\right\vert
\begin{array}{c}
(0,1) \\
(0,\alpha )%
\end{array}%
\right] =[\text{by (1.59) in \cite{MAT}}] \\
&=&x^{-1}{\LARGE H}_{1,1}^{0,1}\left[ \left. \frac{x}{t^{1/\alpha }}%
\right\vert
\begin{array}{c}
(0,\frac{1}{\alpha }) \\
(0,1)%
\end{array}%
\right] =[\text{by (1.60) in \cite{MAT}, for }\sigma =-1] \\
&=&\frac{1}{t^{1/\alpha }}{\LARGE H}_{1,1}^{0,1}\left[ \left. \frac{x}{%
t^{1/\alpha }}\right\vert
\begin{array}{c}
(-\frac{1}{\alpha },\frac{1}{\alpha }) \\
(-1,1)%
\end{array}%
\right] =\frac{1}{t^{1/\alpha }}\frac{1}{2\pi i}\int_{\mathcal{L}}\frac{%
\Gamma (1+\frac{1}{\alpha }-\frac{s}{\alpha })\left( \frac{x}{t^{1/\alpha }}%
\right) ^{-s}ds}{\Gamma (2-s)} \\
&=&\frac{1}{t^{1/\alpha }}\frac{1}{\alpha }{\LARGE
H}_{1,1}^{0,1}\left[ \left. \frac{x}{t^{1/\alpha }}\right\vert
\begin{array}{c}
(1-\frac{1}{\alpha },\frac{1}{\alpha }) \\
(0,1)%
\end{array}%
\right] =[\text{by (1.58) in \cite{MAT}}] \\
&=&\frac{1}{t^{1/\alpha }}\frac{1}{\alpha }{\LARGE
H}_{1,1}^{1,0}\left[ \left. \frac{t^{1/\alpha }}{x}\right\vert
\begin{array}{c}
(1,1) \\
(\frac{1}{\alpha },\frac{1}{\alpha })%
\end{array}%
\right] ,
\end{eqnarray*}%
which coincides, for $t=1,$ with the expression of the stable law
given in (2.28) of \cite{MAT}, in terms of $H$-functions.

On the other hand, for $\alpha =1$, formula (\ref{sta}) reduces to
(\ref{k}) and thus Theorem \ref{Th8} shows that the spectral
density $K(\cdot )$ of the fractional relaxation function $\varphi
_{1,\rho }(\cdot )$ coincides with the density of
$\mathcal{X}_{1,\rho }.$

We now study the dependence structure of the process
$\mathcal{X}_{\alpha ,\rho }$, by means of the auto-codifference
function: adapting the definition given in \cite{SAM} and
\cite{WYL3} to the Laplace transform (since our process is
non-negative valued), we can write, from (\ref{chi}),
\begin{eqnarray}
&&CD(\mathcal{X}_{\alpha ,\rho }(t),\mathcal{X}_{\alpha ,\rho
}(s))
\label{cd} \\
&:&=\ln \mathbb{E}e^{-\eta _{1}\mathcal{X}_{\alpha ,\rho }(t)-\eta _{2}%
\mathcal{X}_{\alpha ,\rho }(s)}-\ln \mathbb{E}e^{-\eta _{1}\mathcal{X}%
_{\alpha ,\rho }(t)}-\ln \mathbb{E}e^{-\eta
_{2}\mathcal{X}_{\alpha ,\rho
}(s)}  \nonumber \\
&=&\ln \frac{\Gamma (\rho ;(\eta _{1}+\eta _{2})^{\alpha }s)\Gamma
(\rho ;\eta _{2}^{\alpha }(t-s))}{\Gamma (\rho ;\eta _{1}^{\alpha
}s)\Gamma (\rho ;\eta _{2}^{\alpha }s)}.  \nonumber
\end{eqnarray}%

By considering the asymptotic behavior (\ref{ttt}) of the
incomplete gamma
function, we easily get%
\begin{equation}
CD(\mathcal{X}_{\alpha ,\rho }(t),\mathcal{X}_{\alpha ,\rho
}(s))\simeq \left( \eta _{1}^{\alpha }-\eta _{2}^{\alpha }\right)
t,  \label{ss}
\end{equation}
for $t\rightarrow +\infty $ and for fixed $s.$ The limiting expression (\ref%
{ss}) coincides with the auto-codifference of the $\alpha $-stable
subordinator. Therefore, we can consider the parameter $\rho $ as
a measure of the local deviation from the temporal dependence
structure displayed in the standard $\alpha $-stable case.

In order to obtain the Cauchy problem satisfied by the
one-dimensional density of $\mathcal{X}_{\alpha ,\rho }$ given in
(\ref{sta}), we distinguish the two cases: $0<\rho \leq 1$ and
$\rho \in (0,1/\alpha ].$

\subsection{Case $0<\protect\rho \leq 1$}

When we restrict to the case $\rho \leq 1$, we can obtain the
integro-differential equation satisfied by the density
(\ref{sta}), in terms of a convolution-type derivative with
space-dependent tail L\'{e}vy measure.

\begin{theorem}
\label{Thmden} Let $p_{\alpha }(\cdot )$ be the law of the $\alpha
$-stable r.v. $S_{\alpha }(1)$, for $\alpha \in (0,1),$ then, for
$\rho \in (0,1),$

\begin{enumerate}
\item
\begin{equation}
\overline{M}_{z}(dw)=\frac{\rho w^{-\frac{1}{\alpha }-\rho
-1}p_{\alpha }(z/w^{1/\alpha })}{\Gamma (1-\rho )}dw,\qquad w>0,
\label{mea}
\end{equation}%
is a L\'{e}vy measure, for any $z>0$.

\item If we define the convolution-type derivative (with
space-dependent
tail L\'{e}vy measure) $_{z}\mathcal{D}_{t}^{\rho }:=\int_{0}^{t}\frac{d}{dt}%
u(t-s)M_{z}(ds),$ where $M_{z}(s)=\int_{s}^{+\infty
}\overline{M}_{z}(dw)$, for any $z>0,$ then the transition density
$h_{\alpha }^{\rho }(x,t)$
satisfies the following integro-differential equation:%
\begin{equation}
\int_{0}^{x}\,_{z}\mathcal{D}_{t}^{\rho }h(x-z,t)dz=-\,^{C}\mathcal{D}%
_{x}^{\alpha \rho }h(x,t),  \label{pp}
\end{equation}%
with the following conditions:%
\begin{equation*}
\left\{
\begin{array}{l}
h(x,0)=\delta (x) \\
\lim_{x\downarrow 0^{+}}h(x,t)=0,%
\end{array}%
\right.
\end{equation*}%
for any $x\geq 0$ and $t\geq 0,$respectively.
\end{enumerate}
\end{theorem}

\begin{proof}
1. We apply the Property 1.2.15 on the tail's behavior of the
stable law, given in \cite{SAM}: $\lim_{x\rightarrow +\infty
}x^{\alpha }P(X>x)=C_{\alpha }\frac{1+\beta }{2}\sigma ^{\alpha
},$ for $C_{\alpha }$ given in (1.2.9), $\beta =1$ (since the
$\alpha $-stable considered here is totally skewed to the right)
and for $\sigma =(\cos (\pi \alpha
/2)w)^{1/\alpha }.$ Thus, by taking the first derivative, we get that $%
p_{\alpha }(x)=O\left( C_{\alpha }^{\prime }wx^{-\alpha -1}\right) $, for $%
x\rightarrow +\infty $ and a constant $C_{\alpha }^{\prime }.$
Thus we can
conclude that $p_{\alpha }(z/w^{1/\alpha })=O\left( Kw^{\frac{1}{\alpha }%
+2}\right) $, for $w\rightarrow 0,$ for a positive constant $K.$
We now
verify that the following finiteness condition is satisfied by $\overline{M}%
_{z}$:%
\begin{equation*}
\int_{0}^{+\infty }(z\wedge 1)\overline{M}_{z}(dz)=\int_{0}^{1}w^{-\frac{1}{%
\alpha }-\rho }p_{\alpha }(z/w^{1/\alpha })dw+\int_{1}^{+\infty }w^{-\frac{1%
}{\alpha }-\rho -1}p_{\alpha }(z/w^{1/\alpha })dw<\infty .
\end{equation*}%

The second integral is finite, since the stable law is finite in
the origin,
while for the first integral we can consider that the integrand function $%
w^{-\frac{1}{\alpha }-\rho }p_{\alpha }(z/w^{1/\alpha })=O\left(
Kw^{2-\rho }\right) $, for $w\rightarrow 0.$

2. We notice that the tail measure $M_{z}(s)$ is finite, for any
$s,z>0$, since $\int_{s}^{+\infty }\frac{\rho w^{-\frac{1}{\alpha
}-\rho -1}p_{\alpha
}(z/w^{1/\alpha })}{\Gamma (1-\rho )}dw<\int_{s}^{+\infty }\frac{\rho w^{-%
\frac{1}{\alpha }-\rho -1}}{\Gamma (1-\rho )}dw<\infty $\ \ Then,
by applying Lemma \ref{Lemrel}, with $\lambda =\eta ^{\alpha }$,
we have that
\begin{equation*}
\frac{\rho \eta ^{\alpha \rho }}{\Gamma (1-\rho )}\int_{0}^{t}\frac{d}{dt}%
\varphi _{\eta ^{\alpha },\rho }(t-s)\Gamma (-\rho ;\eta ^{\alpha
}s)ds=-\eta ^{\alpha \rho }\varphi _{\eta ^{\alpha },\rho }(t),
\end{equation*}%
with $\varphi _{\eta ^{\alpha },\rho }(0)=1$, is satisfied by
(\ref{mlt})$.$
Let now $\mu _{\eta }(s):=\frac{\rho \Gamma (-\rho ;\eta ^{\alpha }s)}{%
\Gamma (1-\rho )}=\eta ^{-\alpha \rho }\nu (s),$ then it is the
tail of a new L\'{e}vy measure denoted by $\overline{\mu }_{\eta
}(s)=\eta ^{-\alpha \rho }\overline{\nu }(s)$ (since it differs
only by a constant from the
original one). Then we get%
\begin{eqnarray*}
&&\frac{\rho }{\Gamma (1-\rho )}\int_{0}^{t}\frac{d}{dt}\varphi
_{\eta
^{\alpha },\rho }(t-s)\Gamma (-\rho ;\eta ^{\alpha }s)ds=\int_{0}^{t}\frac{d%
}{dt}\varphi _{\eta ^{\alpha },\rho }(t-s)\mu _{\eta }(s)ds \\
&=&\eta ^{-\alpha \rho }\int_{0}^{t}\frac{d}{dt}\varphi _{\eta
^{\alpha },\rho }(t-s)\nu _{\eta }(s)ds=-\varphi _{\eta ^{\alpha
},\rho }(t),
\end{eqnarray*}%
which moreover gives, by considering the expression of $\nu _{\eta
}(s),$
\begin{equation*}
\frac{\rho }{\Gamma (1-\rho )}\int_{0}^{t}\frac{d}{dt}\varphi
_{\eta ^{\alpha },\rho }(t-s)\int_{s}^{+\infty }e^{-\eta ^{\alpha
}w}w^{-\rho -1}dwds=-\eta ^{\alpha \rho }\varphi _{\eta ^{\alpha
},\rho }(t).
\end{equation*}%

We can now invert the Laplace transform (in the variable $\eta $)
in the
l.h.s. since $\varphi _{\eta ^{\alpha },\rho }(t)e^{-\eta ^{\alpha }w}=%
\mathcal{L}\left\{ h_{\alpha }^{\rho }(\cdot ,t)\right\}
\mathcal{L}\left\{ p_{\alpha }(\cdot ,w)\right\}
=\mathcal{L}\left\{ h_{\alpha }^{\rho }(\cdot ,t)\ast p_{\alpha
}(\cdot ,w)\right\} ,$ where $\ast $ denotes the
convolution operator:%
\begin{eqnarray*}
&&\frac{\rho }{\Gamma (1-\rho )}\int_{0}^{t}dz\int_{0}^{x}\frac{d}{dt}%
h_{\alpha }^{\rho }(x-z,t-s)\int_{s}^{+\infty }p_{\alpha
}(z,w)w^{-\rho
-1}dwds \\
&=&[\text{by the self-similarity property}] \\
&=&\frac{\rho }{\Gamma (1-\rho )}\int_{0}^{x}dz\int_{0}^{t}\frac{d}{dt}%
h_{\alpha }^{\rho }(x-z,t-s)\int_{s}^{+\infty }p_{\alpha
}(z/w^{1/\alpha })w^{-\frac{1}{\alpha }-\rho -1}dwds,
\end{eqnarray*}%
which coincides with the l.h.s. of (\ref{pp}). The r.h.s. follows
by taking into account the Laplace transform of the Caputo
derivative (\ref{cap}), of order $\alpha \rho <1$, since $\alpha
,\rho <1$, and the second initial
condition, which will be proved below to hold for $h_{\alpha }^{\rho }(x,t)$%
, together with the first one.

Finally, $h_{\alpha }^{\rho }(x,0)=\delta (x),$ by the assumption that $%
\mathcal{X}_{\alpha ,\rho }(0)=0,$ almost surely. On the other
hand, the limit for $x\rightarrow 0^{+}$ follows from the
following considerations:
the density (\ref{sta}) can be rewritten (by (1.60) in \cite{MAT}, with $%
\sigma =-1/\alpha $) as%
\begin{equation*}
h_{\alpha }^{\rho }(x,t)=\frac{\alpha }{t^{1/\alpha }\Gamma (\rho )}{\LARGE H%
}_{1,1}^{0,1}\left[ \left. \frac{x^{\alpha }}{t}\right\vert
\begin{array}{c}
(1-\rho -\frac{1}{\alpha },1) \\
(-1,\alpha )%
\end{array}%
\right] .
\end{equation*}%

The limit for $x\rightarrow 0^{+}$ is then obtained by considering
the
asymptotic behavior of the previous expression (see Theorem 1.3.(\textit{ii}%
) in \cite{MAT}) and since, in this case, $\mu =\alpha -1<0,$ $\delta =\rho +%
\frac{1}{\alpha }-2$ and since $A_{1}-B_{1}=1-\alpha >0$: thus we
get that
\begin{equation*}
h_{\alpha }^{\rho }(x,t)=O\left( x^{-\frac{\alpha \rho +1-\frac{3\alpha }{2}%
}{1-\alpha }}\right) \exp \left\{ -(1-\alpha )\alpha ^{\frac{1}{1-\alpha }%
}x^{-\frac{\alpha }{1-\alpha }}\right\} ,\qquad x\rightarrow
0^{+}.
\end{equation*}
\end{proof}

\begin{remark}
We recall that L\'{e}vy measures expressed in terms of $\alpha
$-stable densities have been already treated in \cite{BEG3}, in
connection with stable processes subordinated by independent
Poisson processes. Moreover, pseudo-differential operators defined
by means of time-dependent L\'{e}vy measures are used in
\cite{ORS}, \cite{BEG2} and \cite{BEG4}, in connection with
time-inhomogeneous subordinators.
\end{remark}

\subsection{Case $1<\protect\rho \leq 1/\protect\alpha $}

For $\rho $ greater than $1$, we derive the fractional
differential equation
satisfied by the transition density, at least in the special case where $%
\rho =2.$ To this aim we recall the definition of the logarithmic
differential operator introduced in \cite{BEG3}: let $\widehat{f}$
denote the Fourier transform, then $\mathcal{P}_{x,c}^{\alpha }$
is the operator
with the following symbol%
\begin{equation}
\widehat{(\mathcal{P}_{x,c}^{\alpha }f)}(\xi )=-\ln (1+c\psi
_{\theta }^{\alpha }(\xi ))\widehat{f}(\xi ),\quad c>0,\text{
}|\xi |<1/c, \label{sy2}
\end{equation}%
for any $f\in \mathcal{L}^{c}.$

\begin{theorem}
Let $\alpha \in (0,1)$ and let $h_{\alpha }^{2}(x,t)$ denote the
transition density of the process defined in (\ref{sta}), for
$\rho =2$, then
\begin{equation}
h_{\alpha }^{2}(x,t)=\frac{\alpha }{x}{\LARGE H}_{1,1}^{0,1}\left[
\left. \frac{x^{\alpha }}{t}\right\vert
\begin{array}{c}
(-1,1) \\
(0,\alpha )%
\end{array}%
\right]  \label{hh4}
\end{equation}%
satisfies the following equation%
\begin{equation}
\frac{\partial }{\partial t}u(x,t)=\mathcal{D}_{x,-\alpha }^{\alpha }u(x,t)-%
\frac{\partial }{\partial t}\mathcal{P}_{x,t}^{\alpha }u(x,t)+\mathcal{P}%
_{x,t}^{\alpha }\left[ \frac{\partial }{\partial t}u(x,t)\right] ,
\label{eq}
\end{equation}%
with initial condition $u(x,0)=\delta (x)$.
\end{theorem}

\begin{proof}
By taking the Fourier transform of (\ref{eq}) and considering (\ref{uno})-(%
\ref{due}) (for $\theta =-\alpha )$ and (\ref{sy2}), we get%
\begin{eqnarray}
&&\frac{\partial }{\partial t}\widehat{u}(\xi ,t)  \label{eq2} \\
&=&\mathcal{-}\psi _{-\alpha }^{\alpha }(\xi )\widehat{u}(\xi ,t)+\frac{%
\partial }{\partial t}\left[ \ln (1+t\psi _{-\alpha }^{\alpha }(\xi ))%
\widehat{u}(\xi ,t)\right] -\ln (1+t\psi _{-\alpha }^{\alpha }(\xi ))\frac{%
\partial }{\partial t}\widehat{u}(\xi ,t)  \nonumber \\
&=&\mathcal{-}\psi _{-\alpha }^{\alpha }(\xi )\widehat{u}(\xi
,t)+\frac{\psi
_{-\alpha }^{\alpha }(\xi )}{1+t\psi _{-\alpha }^{\alpha }(\xi )}\widehat{u}%
(\xi ,t)+\ln (1+t\psi _{-\alpha }^{\alpha }(\xi ))\frac{\partial }{\partial t%
}\widehat{u}(\xi ,t)  \nonumber \\
&&-\ln (1+t\psi _{-\alpha }^{\alpha }(\xi ))\frac{\partial }{\partial t}%
\widehat{u}(\xi ,t)  \nonumber \\
&=&\mathcal{-}\psi _{-\alpha }^{\alpha }(\xi )\widehat{u}(\xi
,t)+\frac{\psi
_{-\alpha }^{\alpha }(\xi )}{1+t\psi _{-\alpha }^{\alpha }(\xi )}\widehat{u}%
(\xi ,t)  \nonumber
\end{eqnarray}%

On the other hand, we evaluate the Fourier transform of
(\ref{sta}) as follows (by considering formulae (2.49) and (2.59)
of the sine and cosine transform of the $H$-function):
\begin{eqnarray}
&&\int_{-\infty }^{+\infty }e^{i\xi x}\frac{\alpha }{x}{\LARGE H}_{1,1}^{0,1}%
\left[ \left. \frac{x^{\alpha }}{t}\right\vert
\begin{array}{c}
(1-\rho ,1) \\
(0,\alpha )%
\end{array}%
\right] dx  \label{fh} \\
&=&\alpha \int_{-\infty }^{+\infty }\cos (\xi x)\frac{1}{x}{\LARGE H}%
_{1,1}^{0,1}\left[ \left. \frac{x^{\alpha }}{t}\right\vert
\begin{array}{c}
(1-\rho ,1) \\
(0,\alpha )%
\end{array}%
\right] dx \nonumber \\
&&+i\alpha \int_{-\infty }^{+\infty }\sin (\xi x)\frac{1}{x}{\LARGE H}%
_{1,1}^{0,1}\left[ \left. \frac{x^{\alpha }}{t}\right\vert
\begin{array}{c}
(1-\rho ,1) \\
(0,\alpha )%
\end{array}%
\right] dx  \nonumber \\
&=&\frac{\alpha \sqrt{\pi }}{2}{\LARGE H}_{3,1}^{0,2}\left[ \left. \frac{%
2^{\alpha }}{t\xi ^{\alpha }}\right\vert
\begin{array}{ccc}
(1,\frac{\alpha }{2}) & (1-\rho ,1) & (\frac{1}{2},\frac{\alpha }{2}) \\
(0,\alpha ) & \; & \;%
\end{array}%
\right] \nonumber \\
&&+\frac{i\alpha \sqrt{\pi }}{2}{\LARGE H}_{3,1}^{0,2}\left[ \left. \frac{%
2^{\alpha }}{t\xi ^{\alpha }}\right\vert
\begin{array}{ccc}
(\frac{1}{2},\frac{\alpha }{2}) & (1-\rho ,1) & (1,\frac{\alpha }{2}) \\
(0,\alpha ) & \; & \;%
\end{array}%
\right]   \nonumber \\
&=&\frac{\alpha \sqrt{\pi }}{2}\frac{1}{2\pi
i}\int_{\mathcal{L}}\left(
\frac{2^{\alpha }}{t\xi ^{\alpha }}\right) ^{-s}\frac{\Gamma (-\frac{\alpha s%
}{2})\Gamma (\rho -s)}{\Gamma (1-\alpha s)\Gamma (\frac{1}{2}+\frac{\alpha s%
}{2})}ds  \nonumber \\
&&+\frac{i \alpha \sqrt{\pi }}{2}\frac{1}{2\pi
i}\int_{\mathcal{L}}\left(
\frac{2^{\alpha }}{t\xi ^{\alpha }}\right) ^{-s}\frac{\Gamma (\frac{1}{2}-%
\frac{\alpha s}{2})\Gamma (\rho -s)}{\Gamma (1-\alpha s)\Gamma (1+\frac{%
\alpha s}{2})}ds  \nonumber \\
&=&[\text{by the duplication formula}]  \nonumber \\
&=&\frac{\alpha }{4}\frac{1}{2\pi i}\int_{\mathcal{L}}\left(
\frac{1}{t\xi ^{\alpha }}\right) ^{-s}\frac{\Gamma (-\frac{\alpha
s}{2})\Gamma (\rho -s)\Gamma (\frac{\alpha s}{2})}{\Gamma
(1-\alpha s)\Gamma (\alpha s)}ds
\nonumber \\
&&+i\alpha \pi \frac{1}{2\pi i}\int_{\mathcal{L}}\left(
\frac{1}{t\xi
^{\alpha }}\right) ^{-s}\frac{\Gamma (-\alpha s)\Gamma (\rho -s)}{\Gamma (-%
\frac{\alpha s}{2})\Gamma (1-\alpha s)\Gamma (1+\frac{\alpha
s}{2})}ds
\nonumber \\
&=&[\text{by the reflection formula}]  \nonumber \\
&=&-\frac{1}{2}\frac{1}{2\pi i}\int_{\mathcal{L}}\left(
\frac{1}{t\xi ^{\alpha }}\right) ^{-s}\frac{\sin (\pi \alpha
s)\Gamma (\rho -s)}{\sin (\pi \alpha s/2)s}ds \nonumber \\
&&+i\frac{1}{2\pi i}\int_{\mathcal{L}}\left( \frac{1}{t\xi
^{\alpha }}\right) ^{-s}\frac{\sin (\pi \alpha s/2)\Gamma (\rho
-s)}{s}ds
\nonumber \\
&=&\frac{1}{2\pi i}\int_{\mathcal{L}}\left( t\xi ^{\alpha }\right) ^{-s}%
\frac{\cos (\pi \alpha s/2)\Gamma (\rho +s)}{s}ds \nonumber \\
&&+\frac{i}{2\pi i}\int_{%
\mathcal{L}}\left( t\xi ^{\alpha }\right) ^{-s}\frac{\sin (\pi
\alpha
s/2)\Gamma (\rho +s)}{s}ds  \nonumber \\
&=&\frac{1}{2\pi i}\int_{\mathcal{L}}\left( e^{-\frac{i\pi \alpha
}{2}}t\xi ^{\alpha }\right) ^{-s}\frac{\Gamma (\rho +s)}{s}ds.
\nonumber
\end{eqnarray}%

In this special case $\rho =2$, we can write (\ref{fh}) as follows%
\begin{eqnarray}
&&\int_{-\infty }^{+\infty }e^{i\xi x}\frac{\alpha }{x}{\LARGE H}_{1,1}^{0,1}%
\left[ \left. \frac{x^{\alpha }}{t}\right\vert
\begin{array}{c}
(1-\rho ,1) \\
(0,\alpha )%
\end{array}%
\right] dx=\frac{1}{2\pi i}\int_{\mathcal{L}}\left( e^{-\frac{i\pi \alpha }{2%
}}t\xi ^{\alpha }\right) ^{-s}(1+s)\Gamma (s)ds  \notag \\
&=&\frac{1}{2\pi i}\int_{\mathcal{L}}\left( e^{-\frac{i\pi \alpha
}{2}}t\xi ^{\alpha }\right) ^{-s}\Gamma (s)ds+\frac{1}{2\pi
i}\int_{\mathcal{L}}\left( e^{-\frac{i\pi \alpha }{2}}t\xi
^{\alpha }\right) ^{-s}\Gamma (s+1)ds  \label{fh2}
\\
&=&[\text{by formula (1.38) in \cite{MAT}}]  \notag \\
&=&e^{-t\xi ^{\alpha }e^{-\frac{i\pi \alpha }{2}}}\left( 1+t\xi
^{\alpha }e^{-\frac{i\pi \alpha }{2}}\right) =e^{-\psi _{-\alpha
}^{\alpha }(\xi )t}(1+t\psi _{-\alpha }^{\alpha }(\xi )).  \notag
\end{eqnarray}%

It is easy to check that (\ref{fh2}) satisfies (\ref{eq2}) as well
as the initial condition.
\end{proof}


\section{Properties of the generalized stable law}\label{sec:4}

\setcounter{section}{4}
\setcounter{equation}{0}\setcounter{theorem}{0}

Let now consider a fix time argument and define
$\mathcal{X}_{\alpha ,\rho }$ as the r.v. with Laplace transform
\begin{equation}
\mathbb{E}e^{-\eta \mathcal{X}_{\alpha ,\rho }}=\frac{\int_{c\eta
^{\alpha }}^{+\infty }e^{-w}w^{\rho -1}dw}{\Gamma \left( \rho
\right) }=\frac{\Gamma \left( \rho ;c\eta ^{\alpha }\right)
}{\Gamma \left( \rho \right) },\qquad \eta >0,  \label{gst}
\end{equation}%
for $\alpha \in (0,1]$, $\rho \leq 1/\alpha $ and scale parameter
$\sigma =c^{1/\alpha }\in \mathbb{R}^{+}$. Then an interesting
particular case can be obtained for $\rho =n\in \mathbb{N}:$ by
successive integrating by parts\
of (\ref{gst}), it can be checked that%
\begin{equation}
\mathbb{E}e^{-\eta \mathcal{X}_{\alpha ,\rho }}=\frac{\int_{c\eta
^{\alpha }}^{+\infty }e^{-w}w^{n-1}dw}{\Gamma \left( n\right)
}=e^{-c\eta ^{\alpha }}\sum_{j=0}^{n-1}\frac{(c\eta ^{\alpha
})^{j}}{j!},  \label{lapp}
\end{equation}%
thus coinciding with the cumulative distribution function of a Poisson r.v. $%
Z_{c\eta ^{\alpha }}$ with parameter $c\eta ^{\alpha }$ (i.e.
$P(Z_{c\eta
^{\alpha }}<n)$). By inverting (\ref{lapp}), the density of $\mathcal{X}%
_{\alpha ,\rho }$ can be written, for $\rho =n>1,$ as follows%
\begin{equation}
h_{\alpha }^{n}(x)=\sum_{j=0}^{n-1}\frac{c^{j}}{j!}\frac{1}{\Gamma
(-\alpha j)}\int_{0}^{x}\frac{p_{\alpha }(z)}{(x-z)^{\alpha
j+1}}dz,  \label{hn}
\end{equation}%
i.e. as a finite sum of convolutions of the standard $\alpha
$-stable density $p_{\alpha }(\cdot ).$

In the case $\rho =2,$ we can also rewrite (\ref{den}) as
\begin{eqnarray}
h_{\alpha }^{2}(x) &=&\frac{\alpha }{x}{\LARGE
H}_{1,1}^{0,1}\left[ \left. \frac{x^{\alpha }}{c}\right\vert
\begin{array}{c}
(-1,1) \\
(0,\alpha )%
\end{array}%
\right] =\frac{\alpha }{x}\frac{1}{2\pi
i}\int_{\mathcal{L}}\frac{\Gamma (2-s)\left( \frac{x^{\alpha
}}{c}\right) ^{-s}}{\Gamma (1-\alpha s)}ds
\notag \\
&=&\frac{\alpha }{x}\frac{1}{2\pi i}\int_{\mathcal{L}}\frac{\Gamma
(z)\left(
\frac{x^{\alpha }}{c}\right) ^{z-2}}{\Gamma (1+\alpha z-2\alpha )}dz=\frac{%
\alpha c^{2}}{x^{2\alpha +1}}W_{-\alpha ,1-2\alpha
}(-\frac{c}{x^{\alpha }}). \notag
\end{eqnarray}%

Moreover, when $\rho =2$ and $\alpha =1/2,$ it reduces to%
\begin{eqnarray*}
h_{1/2}^{2}(x) &=&\frac{c^{2}}{2x^{2}}\frac{1}{2\pi i}\int_{\mathcal{L}}%
\frac{\Gamma (s)\left( \frac{c}{\sqrt{x}}\right) ^{-s}}{\Gamma (s/2)}ds \\
&=&\text{by the duplication formula of the gamma function}] \\
&=&\frac{c^{2}}{4x^{2}\sqrt{\pi }}\frac{1}{2\pi i}\int_{\mathcal{L}}\Gamma (%
\frac{s}{2}+\frac{1}{2})\left( \frac{c}{2\sqrt{x}}\right) ^{-s}ds \\
&=&\frac{c^{2}}{4x^{2}\sqrt{\pi }}{\LARGE H}_{0,1}^{1,0}\left[ \left. \frac{c%
}{2\sqrt{x}}\right\vert
\begin{array}{c}
\, \\
(\frac{1}{2},\frac{1}{2})%
\end{array}%
\right] =[\text{by (1.125) in \cite{MAT}}] \\
&=&\frac{c^{3}}{4x\sqrt{\pi x^{3}}}e^{-\frac{c^{2}}{4x}},
\end{eqnarray*}%
which is linked to a L\'{e}vy distribution $Y_{a}$ with scale parameter $%
a=c^{2}/2$ by the following relationship:%
\begin{equation}
h_{1/2}^{2}(x)=\frac{a}{x}P\left( Y_{a}\in dx\right) .
\label{levy2}
\end{equation}%
which is linked to a L\'{e}vy distribution $Y_{a}$ with scale parameter $%
a=c^{2}/2$ by the following relationship:%
\begin{equation}
h_{1/2}^{2}(x)=\frac{a}{x}P\left( Y_{a}\in dx\right) .
\label{levy2}
\end{equation}%

The Laplace transform can be checked to coincide with (\ref{gst}),
indeed
\begin{equation*}
\int_{c\sqrt{\eta }}^{+\infty }e^{-w}wdw=e^{-c\sqrt{\eta
}}(1+c\sqrt{\eta })
\end{equation*}%
and
\begin{equation*}
\frac{d}{d\eta }\int_{c\sqrt{\eta }}^{+\infty }e^{-w}wdw=\frac{d}{d\eta }%
e^{-c\sqrt{\eta }}(1+c\sqrt{\eta
})=-\frac{c^{2}}{2}e^{-c\sqrt{\eta }}
\end{equation*}%

On the other hand,%
\begin{equation*}
\frac{d}{d\eta }\int_{0}^{+\infty }e^{-\eta x}\frac{a}{x}P\left(
Y_{a}\in
dx\right) =-a\int_{0}^{+\infty }e^{-\eta x}P\left( Y_{a}\in dx\right) =-ae^{-%
\sqrt{2a\eta }}.
\end{equation*}%

Moreover, from (\ref{levy2}) it is evident that the first moment
is finite
and reads%
\begin{equation*}
\mathbb{E}\mathcal{X}_{\frac{1}{2},2}=\frac{c^{2}}{2}.
\end{equation*}%

The tail probabilities of the r.v. $\mathcal{X}_{\alpha ,\rho }$
can be obtained, by means of the Tauberian theorem (see
e.g.\cite{FEL}, Theorem XIII-5-4, p.446), as follows:

\begin{align*}
\int_{0}^{+\infty }e^{-ux}P(\mathcal{X}_{\alpha ,\rho }>x)dx& =\frac{1-%
\mathbb{E}e^{-u\mathcal{X}_{\alpha ,\rho }}}{u} \\
& =\frac{\int_{0}^{cu^{\alpha }}e^{-w}w^{\rho-1}dw}{ u\Gamma
\left(\rho
\right)}=[\text{for }u\rightarrow 0] \\
& \simeq \frac{u^{\alpha \rho -1}}{\Gamma \left( \rho +1\right) },
\end{align*}%
by considering the well-known asymptotics of the (lower)
incomplete gamma
function. Under the condition $\rho <1/\alpha $ we get, from (5.17) in \cite%
{FEL}, for a constant $L(\cdot )$,%
\begin{equation}
\lim_{x\rightarrow +\infty }x^{\alpha \rho }P(\mathcal{X}_{\alpha ,\rho }>x)=%
\frac{1}{\Gamma (\rho +1)\Gamma (1-\alpha \rho )}.  \label{oo}
\end{equation}%

For $\rho =1$ ($\alpha $-stable case), formula (\ref{oo}) reduces
with the well-known result (see formula (1.2.10) in \cite{SAM}).

By taking into account (\ref{oo}) and the fact that
$\mathbb{E}X^{\delta }=\int_{0}^{+\infty }P(X^{\delta }>x)dx,$ for
any positive r.v., we can give the following condition for the
finiteness of the $\delta $-th moments:

\begin{equation*}
\mathbb{E}\left( \mathcal{X}_{\alpha ,\rho }\right) ^{\delta }<+\infty \text{%
, \qquad iff\qquad\ }0<\delta <\alpha \rho .
\end{equation*}%

Therefore, in the case where $\rho <1,$ the r.v.
$\mathcal{X}_{\alpha ,\rho } $ has finite moments only of
fractional order less than $1,$ as happens for the stable law of
order $\alpha \in (0,1).$ For any $\rho $, we can evaluate the
moment of order $\delta \in (0,\alpha \rho ),$ for $\delta \neq
n\in \mathbb{N}$, by applying the Mellin transform formula for the $H$%
-function given in (2.8) of \cite{MAT}:%
\begin{eqnarray}
\mathbb{E}\left( \mathcal{X}_{\alpha ,\rho }\right) ^{\delta } &=&\frac{%
\alpha }{\Gamma (\rho )}\int_{0}^{+\infty }x^{\delta -1}{\LARGE H}%
_{1,1}^{0,1}\left[ \left. \frac{x^{\alpha }}{c}\right\vert
\begin{array}{c}
(1-\rho ,1) \\
(0,\alpha )%
\end{array}%
\right] dx  \label{mom2} \\
&=&\frac{1}{\Gamma (\rho )}\int_{0}^{+\infty }z^{\frac{\delta }{\alpha }-1}%
{\LARGE H}_{1,1}^{0,1}\left[ \left. \frac{z}{c}\right\vert
\begin{array}{c}
(1-\rho ,1) \\
(0,\alpha )%
\end{array}%
\right] dz  \notag \\
&=&\frac{c^{\frac{\delta }{\alpha }}}{\Gamma (\rho )}\frac{\Gamma (\rho -%
\frac{\delta }{\alpha })}{\Gamma (1-\delta )}.  \notag
\end{eqnarray}%
Formula (\ref{mom2}) reduces, for $\rho =1,$ to the well-known
$\delta $-th order moment of the stable law (see \cite{WOL}).


\section*{Acknowledgements}

 The authors are grateful to the referees and the Editor
 for many useful comments on an earlier version and to
 Claudio Macci for insightful discussions and suggestions.





 \bigskip \smallskip

 \it

 \noindent
$^1$ Sapienza University of Rome \\
p.le A.Moro 5 \\
00185 Rome, ITALY  \\[4pt]
  e-mail: luisa.beghin@uniroma1.it (Corr. author)

$^2$  Faculty of Economic Sciences, University of Warsaw \\
Dluga 44/50 \\
00-241 Warsaw, POLAND \\[4pt]
  e-mail: jgajda@wne.uw.edu.pl


\begin{thebibliography}{99} 
 \normalsize 

\bibitem{BEG} L. Beghin, On fractional tempered stable processes and their
governing differential equations. \emph{Journal of Computational
Physics}, \textbf{293} (2015), 29--39

\bibitem{BEG3} L. Beghin, Fractional diffusion-type equations with
exponential and logarithmic differential operators. \emph{Stoch. Proc. Appl.,%
} \textbf{128} (2018), 2427--2447.

\bibitem{BEG2} L. Beghin, C. Ricciuti, Time-inhomogeneous fractional Poisson
processes defined by the multistable subordinator. \emph{Stoch. Anal. Appl.}%
, \textbf{37} No 2 (2019), 171-188.

\bibitem{BEG4} L. Beghin, C. Ricciuti, Pseudo-differential operators and
related additive geometric stable processes, \emph{Markov
Processes and Related Fields}, \textbf{25} (2019), 415-444.

\bibitem{BOC} S. Bochner, \emph{Harmonic Analysis and the Theory of
Probability}. University of California Press, California Monogr.
Math. Sci., Berkeley (1955).

\bibitem{CAR} B.D. Carter, M.D. Springer, The distribution of products,
quotients and powers of independent H-function variates.
\emph{SIAM J. Appl. Math.}, \textbf{33}, No 4 (1977), 542-558.

\bibitem{DIM} I. Dimovsky, V. Kiryakova, The Obrechkoff integral
transform: properties and relation to a generalized fractional
calculus. \emph{Numerical Functional Analysis and Optimization}
\textbf{21}, No 1-2 (2000), 121-144. DOI:
10.1080/01630560008816944.

\bibitem{FEL} W. Feller,\emph{\ An Introduction Probability Theory and its
Applications}. vol.2 (2nd ed.), Wiley, New York (1971).

\bibitem{GAJ} J. Gajda, A. Wylomanska, Time-changed Ornstein--Uhlenbeck
process. \emph{J. Phys. A: Math. Theor}., \textbf{48} (2015) 1-19.

\bibitem{GAJ2} J. Gajda, A. Kumar, A. Wylomanska, Stable L\'{e}vy process
delayed by tempered stable subordinator. \emph{Statistics and
Probability Letters} \textbf{145} (2019) 284--292.

\bibitem{GAR} R. Garra, A. Giusti, F. Mainardi, G. Pagnini, Fractional
relaxation with time-varying coefficient. \emph{Fractional
Calculus and Applied Analysis, }\textbf{17} No 2 (2014), 424--439,
 DOI: 10.2478/s13540-014-0178-0.

\bibitem{GOR} R. Gorenflo, A.A. Kilbas, F. Mainardi, S.V. Rogosin, \emph{%
Mittag-Leffler Functions, Related Topics and Applications.
}Springer-Verlag, Berlin Heidelberg (2014).

\bibitem{JAM} G.J. Jameson, The incomplete gamma functions. \emph{The Math.
Gazette}, \textbf{100} (548), (2016) , 298-306.

\bibitem{KAR} D.B. Karp, J.L. Lopez, Representations of hypergeometric
functions for arbitrary parameter values and their use.
\emph{Journal of Approximation Theory,} \textbf{218} (2017),
42--70.

\bibitem{KAR2} D.B. Karp, J.L. Lopez, An extension of the multiple
Erdelyi-Kober operator and representations of the generalized
hypergeometric functions. \emph{Fract. calcul. Appl. Anal.,}
\textbf{21} No 5 (2018), 1360--1376. DOI: 10.1515/fca-2018-0071.

\bibitem{KIL} A.A. Kilbas, H.M. Srivastava, J.J. Trujillo, \emph{Theory and
Applications of Fractional Differential Equations}. vol. 204 of
North-Holland Mathematics Studies, Elsevier Science B.V.,
Amsterdam (2006).

\bibitem{KIR} V. Kiryakova, \emph{Generalized Fractional
Calculus and Applications}. Pitman Research Notes in Mathematics
Vol. 301, Longman (1994).

\bibitem{KOC} A.N. Kochubei, General fractional calculus, evolution
equations and renewal processes. \emph{Integral Equations Operator Theory, }%
\textbf{71} (2011), 583-600.

\bibitem{KOL} V. Kolokoltsov, The probabilistic point of view on the
generalized fractional partial differential equations.
\emph{Fractional Calculus and Applied Analysis, }\textbf{22}, No
3, (2019), 543--600, DOI: 10.1515/fca-2019-0033.

\bibitem{KUM} A. Kumar, P. Vellaisamy, Inverse tempered stable
subordinators. \emph{Statistics and Probability Letters,
}\textbf{103}, (2015), 134-141.

\bibitem{MAI} F. Mainardi, Fractional relaxation-oscillation and fractional
diffusion-wave phenomena. \emph{Chaos, Solitons and Fractals,}
\textbf{7}, No 9 (1996), 1461-1477.

\bibitem{MAR} O.I. Marichev, \emph{Handbook of Integral Transforms of Higher
Transcendental Functions, Theory and Algorithmic Tables}. Ellis
Horwood, Chichester (1983).

\bibitem{MAT} A.M. Mathai, R.K. Saxena, H.J. Haubold, \emph{The H-Function:
Theory and Applications}. Springer, New York (2010).

\bibitem{MEE} M.M. Meerschaert, A. Sikorskii,\emph{\ Stochastic Models for
Fractional Calculus}. De Gruyter, Berlin/Boston (2012).

\bibitem{MIL} K.S. Miller, B. Ross, \emph{An Introduction to the Fractional
Calculus and Fractional Differential Equations. }Wiley, New York
(1993).

\bibitem{ORS} E. Orsingher, C. Ricciuti, B. Toaldo, Time-inhomogeneous jump
processes and variable order operators. \emph{Potential Analysis,
}\textbf{45}, No 3 (2016), 435-461.

\bibitem{PAR} R.B. Paris, \emph{Incomplete gamma function}. In: F.W. Olver,
D.M. Lozier, R. Boisvert, C.W. Clark, NIST Handbook of
Mathematical Functions, Cambridge University Press (2010).

\bibitem{PIP} V. Pipiras, M. Taqqu, \emph{Long-Range Dependence
and~Self-Similarity}. Cambridge University Press (2017).

\bibitem{SAM} G. Samorodnitsky, M. Taqqu, \emph{Stable Non-Gaussian Random
Processes: Stochastic Models with Infinite Variance. }Chapman and
all, New York (1994).

\bibitem{SCH} R.L. Schilling, R. Song, Z. Vondracek. \emph{Bernstein
Functions: Theory and Applications}, 37, De Gruyter Studies in
Mathematics Series, Berlin (2010).

\bibitem{SPR} M.D. Springer, W.E. Thompson, The distribution of products of
beta, gamma and Gaussian random variables. \emph{SIAM J. Appl.
Math}., \textbf{18}, No 4 (1970), 721-737.

\bibitem{STA} E.W. Stacy, A generalization of the gamma distribution. \emph{%
The Annals of Mathematical Statistics}, \textbf{33}, No 3, (1962),
1187-1192.

\bibitem{STAW} A. Stanislavsky, K. Weron, A. Weron, Diffusion and relaxation
controlled by tempered $\alpha $-stable processes. \emph{Physical
Review E}, 051106 (2008), 1-6.

\bibitem{TOA} B. Toaldo, Convolution-type derivatives, hitting-times of
subordinators and time-changed $C_{0}$-semigroups. \emph{Potential Analysis}%
, \textbf{42}, No 1, (2015), 115-140.

\bibitem{VEL} P. Vellaisamy, K.K. Kataria, The I-function distribution and
its extensions. \emph{Teor. Veroyatnost. i Primenen}., \textbf{63}
No 2, (2018) 284-305.

\bibitem{WYL} A. Wy\l omanska, Arithmetic Brownian motion subordinated by
tempered stable and inverse tempered stable processes.
\emph{Physica A}, \textbf{391} (2012) 5685-5696.

\bibitem{WYL2} A. Wy\l omanska, The tempered stable process with infinitely
divisible inverse subordinators. \emph{J. Stat. Mech. Theory
Exp}., (2013) 1-18.

\bibitem{WYL3} A. Wylomanska, A. Chechkin, J. Gajda, I.M. Sokolov,
Codifference as a~practical~tool to measure
interdependence,~\emph{Physica A: Statistical Mechanics and its
Applications}, \textbf{421}, No 1 (2015), 412-429.

\bibitem{WOL} S.J. Wolfe, On moments of probability distribution
functions.
In: \emph{Fractional Calculus and Its Applications}, B. Ross
(ed.), Lect. Notes in Math. 457, Springer, Berlin, (1975),
306-316.


\end{thebibliography}
\end{document}